\documentclass[reqno]{amsart}

\usepackage{allneeded}

\DeclareMathOperator*{\toup}{\longrightarrow} 
\DeclareMathOperator{\rank}{\mathrm{rank}}
\newcommand{\s}{\ \ \ }
\newcommand{\Line}{\mathds{L}}
\newcommand{\K}{\mathds{K}}

\newcommand{\N}{\mathds{N}}

\newcommand{\R}{\mathds{R}}
\newcommand{\T}{\mathds{T}}

\newcommand{\ev}{\mathrm{ev}}
\newcommand{\id}{\mathrm{id}}
\newcommand{\ind}{\mathrm{ind}}
\newcommand{\avind}{\overline{\ind}}
\newcommand{\nul}{\mathrm{nul}}
\newcommand{\diff}{\mathrm{d}}
\newcommand{\fix}{\mathrm{fix}}
\newcommand{\Hom}{\mathrm{H}}
\newcommand{\Loc}{\mathrm{C}}
\newcommand{\Tan}{\mathrm{T}}
\newcommand{\iso}{I}
\newcommand{\orb}{\mathrm{orb}}
\newcommand{\path}[2]{\Lambda(#1;#2)}
\newcommand{\pathMI}{\path{M}{\iso}}
\newcommand{\mpath}[3]{\Lambda^{#3}(#1;#2)}
\newcommand{\mpathMI}[1]{\mpath{M}{\iso}{#1}}
\newcommand{\W}{W^{1,2}}
\newcommand{\Wloc}{\W_{\mathrm{loc}}}

\begin{document}

\title[On the multiplicity of isometry-invariant geodesics]{On the multiplicity of isometry-invariant  geodesics\\ on product manifolds}

\author{Marco Mazzucchelli}

\address{UMPA, \'Ecole Normale Sup\'erieure de Lyon, 69364 Lyon Cedex 07, France}%
\email{marco.mazzucchelli@ens-lyon.fr}%


\date{May 14, 2012, \emph{Revised:} July 28, 2013.}

\begin{abstract}
We prove that on any closed Riemannian manifold $(M_1\times M_2,g)$, with $\rank\Hom_1(M_1)\neq0$ and $\dim(M_2)\geq2$, every isometry homotopic to the identity admits infinitely many isometry-invariant geodesics. 
\end{abstract}

\maketitle

\begin{quote}
\begin{footnotesize}
\tableofcontents
\end{footnotesize}
\end{quote}

\section{Introduction}

On a closed Riemannian manifold $(M,g)$ equipped with an isometry $\iso$, a natural problem consists in searching for   1-dimensional submanifolds $\ell\looparrowright M$ that are complete geodesics invariant by $\iso$. More precisely, $\ell$ is an \textbf{$\iso$-invariant geodesic} when it can be parametrized with a geodesic $\gamma:\R\to \ell$ of constant positive speed  and there exists $\tau>0$ such that $\iso(\gamma(t))=\gamma(t+\tau)$ for all $t\in\R$. A special instance of this definition is when $\iso$ is the identity, in which case invariant geodesics are precisely closed geodesics. The study of isometry-invariant geodesics was initiated by Grove \cite{Grove:Condition_C_for_the_energy_integral_on_certain_path_spaces_and_applications_to_the_theory_of_geodesics, Grove:Isometry_invariant_geodesics}, who established several existence and  multiplicity results. Further investigations are due to, among others, Grove and Tanaka \cite{Grove_Tanaka:On_the_number_of_invariant_closed_geodesics_BULLETTIN, Grove_Tanaka:On_the_number_of_invariant_closed_geodesics_ACTA}, Tanaka \cite{Tanaka:On_the_existence_of_infinitely_many_isometry_invariant_geodesics}, Hingston \cite{Hingston:Isometry_invariant_geodesics_on_spheres} and Rademacher \cite{Rademacher:Metrics_with_only_finitely_many_isometry_invariant_geodesics}.

While any closed Riemannian manifold possesses closed geodesics, it is not always the case that it possesses isometry-invariant geodesics.  The easiest example is probably the flat torus $\T^2=[0,1]^2/\{0,1\}^2$, on which the rotation $\iso(x,y)=(1-y,x)$ is an isometry without invariant geodesics. Moreover, even if we require an isometry to be isotopic to the identity (e.g.\ the time-1 map of a Killing vector field), it may still have only finitely many invariant geodesics. For instance, on the standard Riemannian sphere $S^2$, any rotation of an angle $\theta$ around  the axis joining north and south poles has only the equator as invariant geodesic unless $\theta$ is a multiple of $\pi$. This is in contrast with the case of closed geodesics: celebrated results due to Bangert \cite{Bangert:On_the_existence_of_closed_geodesics_on_two_spheres}, Franks \cite{Franks:Geodesics_on_S2_and_periodic_points_of_annulus_homeomorphisms}, and Hingston \cite{Hingston:On_the_growth_of_the_number_of_closed_geodesics_on_the_two_sphere} imply that any Riemannian $S^2$ has infinitely many closed geodesics. It is a long standing conjecture in Riemannian geometry that  any closed Riemannian manifold (of dimension at least 2) has infinitely many closed geodesics. In \cite{Tanaka:On_the_existence_of_infinitely_many_isometry_invariant_geodesics}, Tanaka extended a celebrated result by Gromoll and Meyer \cite{Gromoll_Meyer:Periodic_geodesics_on_compact_Riemannian_manifolds} to the setting of isometry-invariant geodesics, asserting that any isometry $\iso$ of a closed simply connected Riemannian manifold $(M,g)$ possesses infinitely many closed geodesics provided the homology of a suitable space of $\iso$-invariant curves is sufficiently rich. By a result of Vigu\'e-Poirrier and Sullivan \cite{ViguePoirrier_Sullivan:The_homology_theory_of_the_closed_geodesic_problem},  this latter assumption is verified if $\iso$ is homotopic to the identity and the cohomology of $M$ is not a truncated polynomial ring in one variable.

In this paper we shall prove the following result that complements the one of Tanaka, and extends another important closed geodesics result due to Bangert and Klingenberg \cite[Corollary~3]{Bangert_Klingenberg:Homology_generated_by_iterated_closed_geodesics}. 

\begin{thm}\label{t:main}
Let $(M,g)$ be a Riemannian manifold such that $M$ is homeomorphic to a product $M_1\times M_2$ of closed manifolds with $\rank \Hom_1(M_1)\neq0$ and $\dim(M_2)\geq2$. Then every isometry $\iso$ of $(M,g)$ that is homotopic to the identity admits infinitely many $\iso$-invariant geodesics.
\end{thm}

The starting point in the study of this kind of results is the  crucial observation due to Grove \cite[Theorem~2.4]{Grove:Isometry_invariant_geodesics} that any isometry with only finitely many invariant geodesics does not have non-closed ones. In view of this fact, the main issue in the proof of multiplicity results consists in identifying several iterations of a same closed isometry-invariant geodesic detected as distinct critical points of an energy function. The classical tools for dealing with this problem, the iteration theory for Morse indices \cite{Bott:On_the_iteration_of_closed_geodesics_and_the_Sturm_intersection_theory, Long:Index_theory_for_symplectic_paths_with_applications} and local homology groups \cite{Gromoll_Meyer:Periodic_geodesics_on_compact_Riemannian_manifolds} of closed geodesics, have been ingeniously extended to the setting of isometry-invariant geodesics by Grove and Tanaka  \cite{Grove_Tanaka:On_the_number_of_invariant_closed_geodesics_BULLETTIN, Grove_Tanaka:On_the_number_of_invariant_closed_geodesics_ACTA, Tanaka:On_the_existence_of_infinitely_many_isometry_invariant_geodesics}. In order to prove our main theorem, we will combine this machinery together with an extension to the isometry-invariant setting of a homological technique of Bangert and Klingenberg \cite{Bangert_Klingenberg:Homology_generated_by_iterated_closed_geodesics}.

Further generalizations of closed geodesics results may be possible. Specifically, remarkable results due to Bangert and Hingston \cite{Bangert_Hingston:Closed_geodesics_on_manifolds_with_infinite_Abelian_fundamental_group} show that every closed Riemannian manifold with infinite abelian fundamental group must have infinitely many closed geodesics. To the best of the author's knowledge, a  generalization of such result to the isometry-invariant case has not been investigated yet. Indeed, even when the fundamental group has rank larger than 1, it is not clear how to conclude that there are infinitely many geodesics invariant by a general isometry homotopic to the identity. We plan to study this problem further  in the future.

\subsection{Organization of the paper} In Section~\ref{s:Preliminaries} we review the variational principle, Morse indices and local homology of isometry-invariant geodesics. In Section~\ref{s:Bangert_lemma}, after recalling Grove and Tanaka's results on the local homology of iterated geodesics, we prove an isometry-invariant version of the homological vanishing result of  Bangert and Klingenberg, and we derive an application to the multiplicity of isometry-invariant geodesics. Finally, in Section~\ref{s:proof_main} we prove Theorem~\ref{t:main}.

\subsection{Acknowledgments} I wish to thank the anonymous referee for his or her careful reading of the manuscript, and for providing useful comments.

\section{Preliminaries}\label{s:Preliminaries}

\subsection{The variational setting}
Throughout this paper,  $(M,g)$ will be a closed Riemannian manifold equipped with an isometry $\iso$. Isometry-invariant geodesics can be detected by the following well-known variational principle. Consider the  path space $\pathMI$ of all $\Wloc$ curves $\zeta:\R\to M$ such that $\iso(\zeta(t))=\zeta(t+1)$ for all $t\in\R$. We recall that a curve $\zeta$ has $\Wloc$-regularity when it is absolutely continuous, weakly differentiable, and the function $t\mapsto g(\dot\zeta(t),\dot\zeta(t))$ is locally integrable. The space $\pathMI$ is a Hilbert manifold, and the tangent space $\Tan_\zeta\pathMI$ is given by all the $\Wloc$ vector fields $X$ along $\zeta$ such that $\iso_*(X(t))=X(t+1)$ for all $t\in\R$. We can   equip $\pathMI$ with a complete Riemannian metric $G$ given by
\begin{align*}
G(X,Y)=\int_0^1 \Big[ g(X(t),Y(t)) + g(\dot X(t),\dot Y(t)) \Big]\diff t,\s\s\forall X,Y\in\Tan_\zeta\pathMI,
\end{align*}
where the dots in this expression denote the covariant derivative along $\zeta$. The energy function $E:\pathMI\to\R$ is given by
\begin{align*}
E(\zeta)=\int_0^1 g(\dot\zeta(t),\dot\zeta(t))\,\diff t.
\end{align*}
This function is $C^\infty$, and satisfies the Palais-Smale condition: any sequence $\{\zeta_n\}\subset\pathMI$ such that $E(\zeta_n)$ is uniformly bounded and $G(\nabla E(\zeta_n),\nabla E(\zeta_n))\to0$ admits a subsequence converging toward a critical point of $E$. The critical points of $E$ are precisely the $g$-geodesics $\gamma:\R\to M$ parametrized with constant speed $g(\dot\gamma,\dot\gamma)$ and such that $\iso(\gamma(t))=\gamma(t+1)$. We refer to reader to Grove \cite{Grove:Condition_C_for_the_energy_integral_on_certain_path_spaces_and_applications_to_the_theory_of_geodesics} for more details about the variational principle associated to isometry-invariant geodesics.

Strictly  speaking, any fixed point of the isometry $\iso$ would be a stationary $\iso$-invariant geodesic. However, in this paper, we will only consider geodesics with positive energy. If $\gamma$ is a critical point of $E$, for each $t_0\in\R$ the translated curve $t\mapsto\gamma(t+t_0)$ is also a critical point of $E$ corresponding to the same geometric curve. If $E(\gamma)>0$, the family $\orb(\gamma)$ of translated curves associated to $\gamma$ forms a critical orbit of $E$. Notice that $\orb(\gamma)\cong S^1$ if $\gamma$ is a periodic curve, otherwise $\orb(\gamma)\cong \R$.

\subsection{Indices of the critical orbits of the energy} Counting the critical orbits of $E$ with positive critical value by means of topological methods does not provide a count for the number of $\iso$-invariant geodesics: a closed $\iso$-invariant geodesic $\ell\cong S^1$ gives rise to infinitely many critical orbits of $E$. Indeed, let $\gamma:\R\to\ell$ be a parametrization of $\ell$ with constant speed, basic period $p\geq1$, and such that $\iso(\gamma(0))=\gamma(1)$. For every real number $k\neq0$ we define $\gamma^k:\R\to M$ to be the curve $\gamma^k(t)=\gamma(kt)$. All the parametrized curves $\gamma^{mp+1}$, where $m\in\N$, belong to $\pathMI$ and are critical points of $E$ corresponding to the same oriented geodesic $\ell$: the curve $\gamma^{mp+1}|_{[0,1]}$ joins $\gamma(0)$ with $\gamma(1)$ after winding around $\ell$ for $m$ times. An essential observation due to Grove \cite[Theorem~2.4]{Grove:Isometry_invariant_geodesics} claims that if $\iso$ has a non-closed invariant geodesic, then it admits uncountably many invariant geodesics. This implies that, in order to study the multiplicity of $\iso$-invariant geodesics, we can assume that all of these geodesics are closed.

A way to identify critical orbits of $E$ corresponding to the same geometric curve is to look at their Morse index, a strategy that was first introduced by Bott \cite{Bott:On_the_iteration_of_closed_geodesics_and_the_Sturm_intersection_theory} in the study of closed geodesics. If $\gamma$ is a critical point of $E$, its \textbf{Morse index} $\ind(\gamma)$ and \textbf{nullity} $\nul(\gamma)$ are defined respectively as the dimension of the negative eigenspace and the dimension reduced by 1 of the kernel of the Hessian of $E$ at $\gamma$. All the curves in the critical orbit of $\gamma$ share the same index and nullity. In \cite{Grove_Tanaka:On_the_number_of_invariant_closed_geodesics_ACTA,Tanaka:On_the_existence_of_infinitely_many_isometry_invariant_geodesics}, Grove and Tanaka developed an  analogue of Bott's theory in the much harder setting of isometry invariant geodesics. For our applications, we will need the following statement.

\begin{prop}[Grove-Tanaka \cite{Grove_Tanaka:On_the_number_of_invariant_closed_geodesics_ACTA,Tanaka:On_the_existence_of_infinitely_many_isometry_invariant_geodesics}]\label{p:average_index}
Let $\gamma$ be a critical point of $E$ that is a periodic curve of basic period $p>0$. Then $m^{-1}\ind(\gamma^{mp+1})$ converges to a non-negative number $\avind(\gamma)$ as $m\to\infty$. Moreover, if $\avind(\gamma)=0$ then $\ind(\gamma^{mp+1})=0$ for all $m\in\N$. \hfill\qed
\end{prop}

Another important index of homological nature is the \textbf{local homology} of $E$ at $\orb(\gamma)$, defined as the relative homology group
\begin{align*}
\Loc_*(E,\orb(\gamma))
:=
\Hom_*(\{E<c\}\cup\orb(\gamma),\{E<c\}),
\end{align*}
where $c=E(\gamma)$, and $\Hom_*$ denotes the singular homology functor with rational coefficients. The interplay between the local homology and the Morse indices of a critical orbit can be summarized by saying that the graded group $\Loc_*(E,\orb(\gamma))$ is always trivial in degree less than $\ind(\gamma)$ or greater than $\ind(\gamma)+\nul(\gamma)+1$. Local homology groups are the ``building blocks'' for the homology of the path space $\pathMI$. More precisely, for all $b>c=E(\gamma)$ such that the interval $(c,b)$ does not contain critical values of $E$, the inclusion induces an injective homomorphism
\begin{align*}
\Loc_*(E,\orb(\gamma))
\hookrightarrow
\Hom_*(\{E<b\},\{E<c\}).
\end{align*}
Here, we also allow $b$ to be equal to $+\infty$, in which case $\{E<b\}=\pathMI$. For a general interval $[a,b]\subset(0,+\infty]$ and any homological degree $d$, we have the \textbf{Morse inequality}
\begin{align*}
\rank \Hom_d(\{E<b\},\{E<a\}) \leq \sum_{\orb(\gamma)} \rank \Loc_d(E,\orb(\gamma)),
\end{align*}
where the sum on the right-hand side runs over all the critical orbits $\orb(\gamma)$ such that $a\leq E(\gamma)<b$.

\section{Bangert-Klingenberg Lemmas for isometry-invariant geodesics}\label{s:Bangert_lemma}

In their seminal paper \cite{Bangert_Klingenberg:Homology_generated_by_iterated_closed_geodesics}, Bangert and Klingenberg showed that any sufficiently iterated closed geodesic with average Morse index 0, and that is not a global minimum of the energy in his free homotopy class, cannot arise as a minimax point generated by a (relative) homology class of the free loop space. The important consequence of this result is that, whenever there is a closed geodesic  with these properties that is homologically visible, the Riemannian manifold must contain infinitely many closed geodesics. The proof of this result is based on a homotopic technique introduced earlier by Bangert \cite{Bangert:Closed_geodesics_on_complete_surfaces}, and further employed in different settings by Bangert-Hingston \cite{Bangert_Hingston:Closed_geodesics_on_manifolds_with_infinite_Abelian_fundamental_group}, Hingston \cite{Hingston:On_the_growth_of_the_number_of_closed_geodesics_on_the_two_sphere}, Bangert \cite{Bangert:On_the_existence_of_closed_geodesics_on_two_spheres}, Long \cite{Long:Multiple_periodic_points_of_the_Poincare_map_of_Lagrangian_systems_on_tori}, Lu \cite{Lu:The_Conley_conjecture_for_Hamiltonian_systems_on_the_cotangent_bundle_and_its_analogue_for_Lagrangian_systems}, and the author \cite{Mazzucchelli:The_Lagrangian_Conley_conjecture, Mazzucchelli:On_the_multiplicity_of_non_iterated_periodic_billiard_trajectories}. In this section  we apply results due to Grove and Tanaka \cite{Grove_Tanaka:On_the_number_of_invariant_closed_geodesics_BULLETTIN, Grove_Tanaka:On_the_number_of_invariant_closed_geodesics_ACTA, Tanaka:On_the_existence_of_infinitely_many_isometry_invariant_geodesics} in order to establish the analogue of Bangert and Klingenberg's result in the context of closed isometry-invariant geodesics.

\subsection{Local homology of iterated orbits}\label{s:iterated_local_homology}
In the setting of Section~\ref{s:Preliminaries}, let $\gamma$ be a critical point of $E$  that is a periodic curve of basic period $p\geq1$ and average index $\avind(\gamma)=0$. Since in this paper we are looking for infinitely many $\iso$-invariant geodesics, we can assume that each $\orb(\gamma^{mp+1})$, where $m\in\N$, is isolated in the set of critical points of $E$. In~\cite{Grove_Tanaka:On_the_number_of_invariant_closed_geodesics_ACTA, Tanaka:On_the_existence_of_infinitely_many_isometry_invariant_geodesics}, Grove and Tanaka showed that, up to isomorphism, there are only finitely many different groups  in the family $\{\Loc_*(E,\gamma^{pm+1})\ |\ m\in\N\}$, a statement established earlier by Gromoll and Meyer \cite{Gromoll_Meyer:Periodic_geodesics_on_compact_Riemannian_manifolds}  in case $\iso=\id$. For later purposes, we need to rephrase their results, and we refer to the reader to their papers for a detailed proof.

For syntactic convenience, for every $\tau>0$ let us consider the Hilbert manifold
\[
\mpathMI{\tau}=\left\{\gamma\in\Wloc(\R;M)\ \Big|\ \iso(\gamma(t))=\gamma(t+\tau)\ \ \forall t\in\R\right\},
\]
and the energy function $E^\tau:\mpathMI{\tau}\to\R$ defined by
\begin{align*}
E^\tau(\zeta)=\frac{1}{\tau} \int_0^\tau g(\dot\zeta(t),\dot\zeta(t))\,\diff t.
\end{align*}
There is an obvious diffeomorphism of Hilbert manifolds $\Psi^\tau:\mpathMI{\tau}\to\pathMI$ given by $\Psi^\tau(\zeta)=\zeta^\tau$, where $\zeta^\tau(t)=\zeta(\tau t)$. Moreover $E\circ\Psi^\tau=\tau^2 E^\tau$. With this notation, for each $m\in\N$, the curve $\gamma$ is the critical point of $E^{mp+1}$ corresponding to the critical point $\gamma^{mp+1}$ of $E$. If $c=E^{mp+1}(\gamma)$, we denote the local homology of $E^{mp+1}$ at $\orb(\gamma)$ by 
\[
\Loc_*(E^{mp+1},\orb(\gamma)):=
\Hom_*(\{E^{mp+1}<c\}\cup\orb(\gamma),\{E^{mp+1}<c\}).
\]
The diffeomorphism $\Psi^{mp+1}$ induces a homology isomorphism of this group with the local homology $\Loc_*(E,\orb(\gamma^{mp+1}))$.

We first consider the case in which the period $p$ of $\gamma$ is irrational, treated in~\cite[Section~3]{Tanaka:On_the_existence_of_infinitely_many_isometry_invariant_geodesics}. For all $\mu\in\N$, we define the Hilbert manifold
\begin{align*}
\Lambda^{m,\mu}
:= &\,
\mpathMI{mp+1}\cap\mpath{M}{\mathrm{id}}{\mu p}\\
= &\,
\big\{
\zeta\in\Wloc(\R;M)\ \big|\ \iso(\zeta(t-mp-1))=\zeta(t)=\zeta(t+\mu p)\ \ \forall t\in\R
\big\}.
\end{align*}
It readily follows from the definition that $\Lambda^{m,\mu}=\Lambda^{n,\mu}$ if $m\equiv n \mod\mu$. Moreover, for all $\zeta\in\Lambda^{m,\mu}\cap C^\infty(\R;M)$, the function $t\mapsto g(\dot\zeta(t),\dot\zeta(t))$ is constant, being both $\mu p$-periodic and $(mp+1)$-periodic with $(\mu p)^{-1}(mp+1)$ irrational. Since  smooth curves are dense in $\Lambda^{m,\mu}$, if we denote by\footnote{In some sense, the map $\iota^{m,\mu}$, as well as the map $j^{\mu,\tau,\alpha,\theta}$ introduced below before Lemma~\ref{l:rational}, plays the same role as the iteration map in the theory of closed geodesics.} $\iota^{m,\mu}:\Lambda^{m,\mu}\hookrightarrow\mpathMI{mp+1}$ the inclusion of the corresponding spaces, we have that $E^{mp+1}\circ\iota^{m,\mu}=E^{\mu p}|_{\Lambda^{m,\mu}}$. According to the following lemma, which is a variation of \cite[Lemma~3.2]{Tanaka:On_the_existence_of_infinitely_many_isometry_invariant_geodesics}, some of the maps  $\iota^{m,\mu}$ induce an isomorphism of the corresponding local homology groups of $\orb(\gamma)$.

\begin{lem}\label{l:irrational}
There exists a bounded function $\mu:\N\to\N$ such that, for all but finitely many $m\in\N$, the inclusion $\iota^{m,\mu(m)}$ induces a  homology isomorphism
\[
\iota^{m,\mu(m)}_*:\Loc_*(E^{\mu(m)p}|_{\Lambda^{m,\mu(m)}},\orb(\gamma))\toup^{\cong}\Loc_*(E^{mp+1},\orb(\gamma)).
\]
\end{lem}

\begin{proof}
Throughout this proof, let us adopt the extensive notation $\ind(F,x)$ and $\nul(F,x)$ to denote Morse index and nullity of a function $F$ at a critical point $x$. We recall that our curve $\gamma$ is supposed to have average Morse index 0. By Proposition~\ref{p:average_index}, $\ind(E^{mp+1},\gamma)=0$ for all $m\in\N$. Since $\ind(E^{\mu p}|_{\Lambda^{m,\mu}},\gamma)\leq\ind(E^{mp+1},\gamma)$, we infer
\[
\ind(E^{\mu p}|_{\Lambda^{m,\mu}},\gamma)=\ind(E^{mp+1},\gamma)=0,\s\s\forall \mu,m\in\N.
\]

In the proof of \cite[Lemma~3.2]{Tanaka:On_the_existence_of_infinitely_many_isometry_invariant_geodesics}, Tanaka showed that there exists a bounded function $\mu:\N\to\N$ such that, for all $m\in\N$ large enough, the null spaces of the Hessians of $E^{\mu(m)p}|_{\Lambda^{m,\mu}}$ and $E^{mp+1}$ at $\gamma$ are the same. In particular, there exists $m_0\in\N$ such that 
\[
\nul(E^{\mu(m)p}|_{\Lambda^{m,\mu}},\gamma)=\nul(E^{mp+1},\gamma),\s\s\forall m\geq m_0.
\]
Now, let us equip the Hilbert manifold $\mpathMI{mp+1}$ with the Riemannian metric 
\begin{align}\label{e:Riemannian_metric_mp+1}
G^{mp+1}(X,Y)=\int_0^{mp+1} \Big[ g(X(t),Y(t)) + g(\dot X(t),\dot Y(t)) \Big]\diff t,\\\forall X,Y\in\Tan_\zeta\mpathMI{mp+1}.
\end{align}
We denote by $\nabla E^{mp+1}$ the gradient of the energy $E^{mp+1}$ with respect to this Riemannian metric. A standard computation shows that, if $\zeta\in \mpathMI{mp+1}$ is periodic with period $q$, then $\nabla E^{mp+1}(\zeta)$ is a $q$-periodic vector field along $\zeta$. In particular $\nabla E^{mp+1}(\zeta)$ belongs to the tangent space $\Tan_\zeta \Lambda^{m,\mu}$ for all $\zeta\in\Lambda^{m,\mu}$.

Summing up, for all $m\geq m_0$, the submanifold $\Lambda^{m,\mu(m)}$ of  $\mpathMI{mp+1}$ is invariant under the gradient flow of $\mpathMI{mp+1}$, and the Morse index and nullity of $E^{mp+1}$ at $\gamma$ do not change when we restrict the function to the submanifold  $\Lambda^{m,\mu(m)}$. By a standard argument in Morse theory (see e.g.~\cite[Theorem~5.1.1]{Mazzucchelli:Critical_point_theory_for_Lagrangian_systems}), our statement follows.
\end{proof}

Let us now consider the case in which the period $p$ is rational, for which the reference is \cite[Sections~2--3]{Grove_Tanaka:On_the_number_of_invariant_closed_geodesics_ACTA}. Let $a$ and $b$ be relatively prime positive integers such that $p=a/b$. Notice that $\iso^a(\gamma(t))=\gamma(t)$ for all $t\in\R$, namely the geometric curve $\gamma(\R)$ is contained in $\fix(\iso^a)$. We recall that the fixed points set of an isometry is a collection of closed totally-geodesic submanifolds of $(M,g)$. For a fixed value of $m$, consider three positive integers $\tau$, $\alpha$ and $\theta$ with the following properties:
\begin{itemize}
\item $\tau^{-1}(mp+1)$ is a positive integer,
\item $a$ divides $\alpha$, in particular $\iso^\alpha(\gamma(t))=\gamma(t)$ for all $t\in\R$,
\item $\iso^\theta(\gamma(t))=\gamma(t+\tau)$ for all $t\in\R$,
\item $\tau^{-1}\theta(mp+1)\equiv1\mod \alpha$.
\end{itemize}
These properties readily imply that $\gamma\in\mpath{\fix(\iso^\alpha)}{\iso^\theta}{\tau}$ and there is an inclusion $j^{m,\tau,\alpha,\theta}:\mpath{\fix(\iso^\alpha)}{\iso^\theta}{\tau}\hookrightarrow\mpathMI{mp+1}$ that is a smooth embedding of Hilbert manifolds. Moreover $E^{mp+1}\circ j^{m,\tau,\alpha,\theta}=E^\tau|_{\mpath{\fix(\iso^\alpha)}{\iso^\theta}{\tau}}$. The following Lemma is a variation of \cite[Lemma~2.9]{Grove_Tanaka:On_the_number_of_invariant_closed_geodesics_ACTA} together with \cite[Proposition~3.6]{Grove_Tanaka:On_the_number_of_invariant_closed_geodesics_ACTA}.

\begin{lem}
\label{l:rational}
There exist  bounded functions $\tau:\N\to\N$, $\alpha:\N\to\N$ and $\theta:\N\to\N$ such that, for all $m\in\N$, the inclusion $j^{m,\tau(m),\alpha(m),\theta(m)}$ induces a  homology isomorphism
\[
j^{m,\tau(m),\alpha(m),\theta(m)}_*:\Loc_*(E^{\tau(m)}|_{\mpath{\fix(\iso^{\alpha(m)})}{\iso^{\theta(m)}}{\tau(m)}},\orb(\gamma))\toup^{\cong}\Loc_*(E^{mp+1},\orb(\gamma)).
\]
\end{lem}

\begin{proof}
We proceed as in the proof of Lemma~\ref{l:irrational}. For any given $m\in\N$, if  $\tau$, $\alpha$ and $\theta$ are integers as above, our assumption on the average Morse index of $\gamma$ implies that
\[
\ind(E^\tau|_{\mpath{\fix(\iso^\alpha)}{\iso^\theta}{\tau}},\gamma)=
\ind(E^{mp+1},\gamma)=0.
\]
By \cite[Lemma~2.9]{Grove_Tanaka:On_the_number_of_invariant_closed_geodesics_ACTA}, there exist 
 bounded functions $\tau:\N\to\N$, $\alpha:\N\to\N$ and $\theta:\N\to\N$ such that, for all $m\in\N$, we have
\[
\nul(E^{\tau(m)}|_{\mpath{\fix(\iso^{\alpha(m)})}{\iso^{\theta(m)}}{\tau(m)}},\gamma)=
\nul(E^{mp+1},\gamma)=0,\s\s\forall m\in\N.
\]
Let $G^{mp+1}$ be the standard Riemannian metric on $\mpathMI{mp+1}$, already introduced in~\eqref{e:Riemannian_metric_mp+1},  and $\nabla E^{mp+1}$ the gradient of the energy $E^{mp+1}$ with respect to this Riemannian metric. Since $\fix(\iso^{\alpha(m)})$ is a collection of totally-geodesic submanifolds of $M$, the proof of~\cite[Proposition~3.5]{Grove_Tanaka:On_the_number_of_invariant_closed_geodesics_ACTA} shows that $\nabla E^{mp+1}(\zeta)$ is tangent to $\mpath{\fix(\iso^{\alpha(m)})}{\iso}{mp+1}$ for all $\zeta\in\mpath{\fix(\iso^{\alpha(m)})}{\iso}{mp+1}$. Moreover, a standard computation shows that $\nabla E^{mp+1}(\zeta)$ is tangent to $\mpath{M}{\iso^{\theta(m)}}{\tau(m)}$ for all $\zeta\in\mpath{M}{\iso^{\theta(m)}}{\tau(m)}\cap\mpathMI{mp+1}$. Since 
\[
\mpath{\fix(\iso^{\alpha(m)})}{\iso^{\theta(m)}}{\tau(m)}
=
\mpath{\fix(\iso^{\alpha(m)})}{\iso}{mp+1}
\cap
\mpath{M}{\iso^{\theta(m)}}{\tau(m)},
\]
we conclude that $\mpath{\fix(\iso^{\alpha(m)})}{\iso^{\theta(m)}}{\tau(m)}$ is a submanifold of $\mpathMI{mp+1}$ that is invariant by the gradient flow of $E^{mp+1}$, and the Morse index and nullity of $\gamma$ do not change when we restrict $E^{mp+1}$ to this submanifold. As in the proof of Lemma~\ref{l:irrational}, a standard argument in Morse theory implies our statement.
\end{proof}

\subsection{Bangert's construction}\label{s:bangert_construction}

In the proof of the main results of this section (Lemmas \ref{l:bangert}--\ref{l:bangert_local_homology} and Proposition~\ref{p:bangert_klingenberg}) we will need a homotopy constructed by Bangert in \cite{Bangert:Closed_geodesics_on_complete_surfaces}, that we shall now review with our notation. Fix a period $p\in\N$, and consider a smooth path 
\[\Gamma:[a,b]\to\mpath{M}{\id}{p},\] 
i.e.\ each curve $t\mapsto\Gamma(s)(t)$ is $p$-periodic. For each $m\in\N$, we define an associated continuous path 
\begin{align}\label{e:bangert_Gamma_above}
\Gamma\!_{\langle m\rangle}:[a,b]\to\mpath{M}{\id}{mp} 
\end{align}
in the following way. For the sake of simplicity,  let us assume that $[a,b]=[0,1]$. For each $s\in[0,1]$, we set
\small
\begin{align*}
\Gamma\!_{\langle m\rangle}(\tfrac{s}{m})(t) & = \Gamma(s)(\tfrac{2t}{2-s}), & \forall t\in[0,(1-\tfrac{s}{2})p],\\
\tag*{$(\star)$}
\Gamma\!_{\langle m\rangle}(\tfrac{s}{m})(t) & = \Gamma(4-s-\tfrac{4}{p}t)(p), & \forall t\in[(1-\tfrac{s}{2})p,(1-\tfrac{s}{4})p],\\
\Gamma\!_{\langle m\rangle}(\tfrac{s}{m})(t) & = \Gamma(0)(t+\tfrac{s}{4}p), & \forall t\in[(1-\tfrac{s}{4})p,(m-\tfrac{s}{4})p],\\
\tag*{$(\star)$}
\Gamma\!_{\langle m\rangle}(\tfrac{s}{m})(t) & = \Gamma(\tfrac{4}{p}t + s - 4m)(mp), & \forall t\in[(m-\tfrac{s}{4})p,mp],
\end{align*}
\normalsize
see Figure~\ref{f:bangert}(b). For each $s\in[0,1]$ and $k\in\{1,...,m-2\}$, we set 
\small
\begin{align*}
\Gamma\!_{\langle m\rangle}(\tfrac{k+s}{m})(t) & = \Gamma(1)(2t), & \forall t\in[0,\tfrac{p}{2}],\\
\Gamma\!_{\langle m\rangle}(\tfrac{k+s}{m})(t) & = \Gamma(1)(t-\tfrac{p}{2}) & \forall t\in[\tfrac{p}{2},(k-\tfrac{1}{2})p],\\
\tag*{$(\star)$}
\Gamma\!_{\langle m\rangle}(\tfrac{k+s}{m})(t) & = \Gamma(4k-1-\tfrac{4}{p}t)(kp), & \forall t\in[(k-\tfrac{1}{2})p,(k-\tfrac{1+s}{4})p],\\
\Gamma\!_{\langle m\rangle}(\tfrac{k+s}{m})(t) & = \Gamma(s)(t-\tfrac{1+s}{4}p), & \forall t\in[(k-\tfrac{1+s}{4})p,(k+\tfrac{3-s}{4})p],\\
\tag*{$(\star)$}
\Gamma\!_{\langle m\rangle}(\tfrac{k+s}{m})(t) & = \Gamma(4k+3-\tfrac{4}{p}t)((k+1)p), & \forall t\in[(k+\tfrac{3-s}{4})p,(k+\tfrac{3}{4})p],\\
\Gamma\!_{\langle m\rangle}(\tfrac{k+s}{m})(t) & = \Gamma(0)(t+\tfrac{1}{4}p), & \forall t\in[(k+\tfrac{3}{4})p,(m-\tfrac{1}{4})p],\\
\tag*{$(\star)$}
\Gamma\!_{\langle m\rangle}(\tfrac{k+s}{m})(t) & = \Gamma(\tfrac{4}{p}t-4m+1)(mp), & \forall t\in[(m-\tfrac{1}{4})p,mp],
\end{align*}
\normalsize
see Figure~\ref{f:bangert}(c). Finally, for each $s\in[0,1]$, we set
\small
\begin{align*}
\Gamma\!_{\langle m\rangle}(\tfrac{m-1+s}{m})(t) & = \Gamma(1)(\tfrac{2}{1+s}t), & \forall t\in[0,\tfrac{1+s}{2}p],\\
\Gamma\!_{\langle m\rangle}(\tfrac{m-1+s}{m})(t) & = \Gamma(1)(t+\tfrac{1-s}{2}p) & \forall t\in[\tfrac{1+s}{2}p,(m+\tfrac{s-3}{2})p],\\
\tag*{$(\star)$}\Gamma\!_{\langle m\rangle}(\tfrac{m-1+s}{m})(t) & = \Gamma(2s+4m-5-\tfrac{4}{p}t)((m-1)p), & \forall t\in[(m+\tfrac{s-3}{2})p,(m+\tfrac{s-5}{4})p],\\
\Gamma\!_{\langle m\rangle}(\tfrac{m-1+s}{m})(t) & = \Gamma(s)(t+\tfrac{1-s}{4}p), & \forall t\in[(m+\tfrac{s-5}{4})p,(m+\tfrac{s-1}{4})p],\\
\tag*{$(\star)$}
\Gamma\!_{\langle m\rangle}(\tfrac{m-1+s}{m})(t) & = \Gamma(\tfrac{4}{p}t+1-4m)(mp), & \forall t\in[(m+\tfrac{s-1}{4})p,mp],
\end{align*}
\normalsize
see Figure~\ref{f:bangert}(d). 
\begin{figure}[p]
\footnotesize
\centering
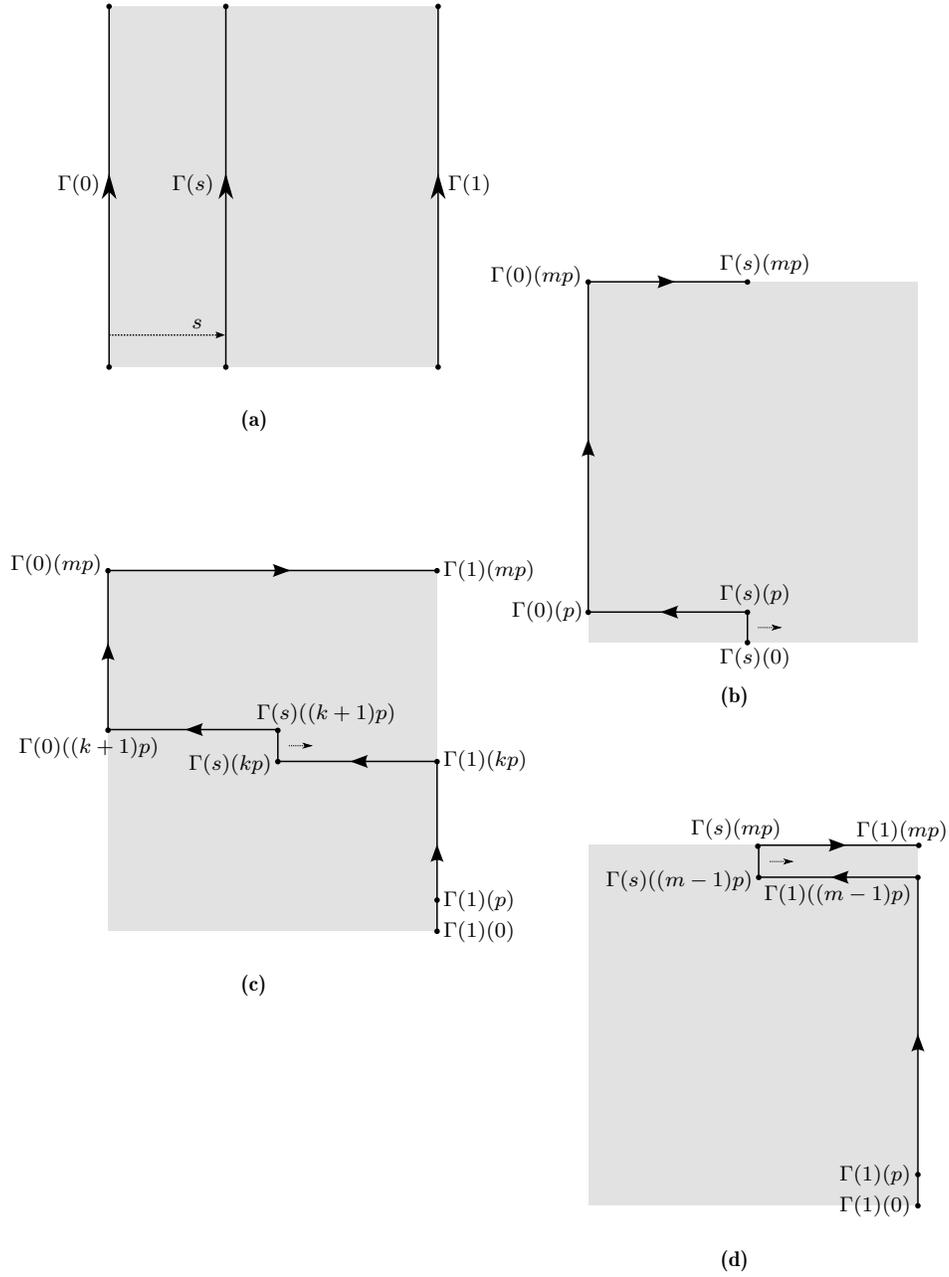 
\normalsize 
\captionstyle{myHang}
\caption{\textbf{(a)} Path $\Gamma:[0,1]\to\mpath{M}{\id}{p}$; the shaded square contains the portion of each curve $t\mapsto\Gamma(s)(t)$ for $t\in[0,mp]$. \textbf{(b--d)} Curve $t\mapsto\Gamma\!_{\langle m\rangle}(s')(t)$ for $t\in[0,mp]$ and suitable values of $s'$; the little horizontal arrows show the direction in which the curve is pulled as $s'$ grows.}
\label{f:bangert}
\end{figure}
If the original path $\Gamma$ is only continuous, the curves $s\mapsto\Gamma(s)(t)$ are continuous as well, but not $\W_{\mathrm{loc}}$. However, we can still define a continuous path $\Gamma\!_{\langle m\rangle}:[a,b]\to\mpath{M}{\id}{mp}$ in a similar way.  All we have to do is modify the equations marked with $(\star)$ in the construction. Each of those equations is of the form
\begin{equation}\label{e:bangert_technicality}
\Gamma\!_{\langle m\rangle}(s')(t)  = \Gamma((-1)^i \tfrac{4}{p} t + s'')(t'), \s\s \forall t\in[t_0,t_1],
\end{equation}
for suitable values $i\in\{1,2\}$, $s',s''\in[0,1]$, $t'\in[0,mp]$, and an interval $[t_0,t_1]\subset[0,mp]$. Let $\delta=\delta(\Gamma)>0$ be such that, for all $r\in\R$ and $r_0,r_1\in[0,1]$ with $|r_0-r_1|\leq\tfrac{4}{p}\delta$, the distance between the points $\Gamma(r_0)(r)$ and $\Gamma(r_1)(r)$ is less than the injectivity radius of $(M,g)$. We set $\lambda:[t_0,t_1]\to M$ to be the piecewise smooth curve such that
\begin{itemize}
\item $t\mapsto g(\dot\lambda(t),\dot\lambda(t))$ is constant,

\item for each non-negative integer $k$ such that $(k+1)\delta<|t_1-t_0|$, the curve $\lambda|_{[t_0+k\delta,t_0+(k+1)\delta]}$ is the (unique) length-minimizing geodesic joining the points $\Gamma((-1)^i \tfrac{4}{p} (t_0+k\delta) + s'')(t')$ and $\Gamma((-1)^i \tfrac{4}{p} (t_0+(k+1)\delta) + s'')(t')$.

\item if $k$ is the maximal non-negative integer such that $k\delta<|t_1-t_0|$, the curve $\lambda|_{[t_0+k\delta,t_1]}$ is the (unique) length-minimizing geodesic joining the points $\Gamma((-1)^i \tfrac{4}{p} (t_0+k\delta) + s'')(t')$ and $\Gamma((-1)^i \tfrac{4}{p} t_1 + s'')(t')$.
\end{itemize}
In the definition of $\Gamma\!_{\langle m\rangle}$, we replace equation~\eqref{e:bangert_technicality} with
\[
\Gamma\!_{\langle m\rangle}(s')(t)  = \lambda(t), \s\s \forall t\in[t_0,t_1].
\]
In this way, $\Gamma\mapsto\Gamma\!_{\langle m\rangle}$ defines a continuous map of the form 
\[C^0([a,b];\mpath{M}{\id}{p})\to C^0([a,b];\mpath{M}{\id}{mp}).\]

Notice that, if we see the path $\Gamma$ as a path in $\mpath{M}{\id}{mp}$ via the inclusion $\mpath{M}{\id}{p}\hookrightarrow\mpath{M}{\id}{mp}$, then $\Gamma$ is homotopic to $\Gamma\!_{\langle m\rangle}$ with fixed endpoints. A possible homotopy $h_\Gamma:[0,1]\times[a,b]\to\mpath{M}{\id}{mp}$ can be defined by setting 
\begin{align}\label{e:homotopy_path_bangert_path}
h_\Gamma(r,s)=
\left\{
  \begin{array}{lcl}
    \Gamma(s) & &\mbox{ if } s\in[r,1],  \\ \\
    (\Gamma|_{[0,r]})_{\langle m\rangle}(s) & &\mbox{ if } s\in[0,r].
  \end{array}
\right.
\end{align}

The main property of this construction is that, by taking $m$ large, the energy of each $mp$-periodic curve $\Gamma\!_{\langle m\rangle}(s)$ can be made almost as small as the energy of $\Gamma(0)$ and $\Gamma(1)$. More precisely
\small
\begin{equation}\label{e:estimate_bangert}
\begin{split}
E^{mp}(\Gamma\!_{\langle m\rangle}(s)) 
& = \frac{1}{mp} \int_0^{mp} g(\tfrac{\diff}{\diff t}\Gamma\!_{\langle m\rangle}(s)(t),\tfrac{\diff}{\diff t}\Gamma\!_{\langle m\rangle}(s)(t))\,\diff t\\
& \leq \frac{1}{mp} \left( (m-2)\max_{s'\in\{0,1\}}\left\{
\int_0^{p} g(\tfrac{\diff}{\diff t}\Gamma(s')(t),\tfrac{\diff}{\diff t}\Gamma(s')(t))\,\diff t
\right\} + C_\Gamma \right)\\
& < \max \left\{E^{p}(\Gamma(0)),E^{p}(\Gamma(1)) \right\} + \frac{C_\Gamma}{mp},
\end{split}
\end{equation}
\normalsize
where $C_\Gamma>0$ is a constant depending continuously on $\Gamma$, but not on $m$.

\subsection{Bangert-Klingenberg lemmas}
Let $\K\subset\N$ be an infinite subset, $p,p_0\in\N$, and $\Omega$ a Hilbert manifold contained in
\[\mpath{M}{\id}{p_0}\cap\displaystyle\bigcap_{m\in\K} \mpathMI{mp+1}\]
as a Hilbert submanifold of each space involved in the intersection. Assume also that the energy functions $E^{p_0}$ and $E^{mp+1}$, for all $m\in\K$, coincide on $\Omega$, i.e.
\begin{align*}
E^{p_0}(\zeta)
=\frac{1}{p_0}\int_0^{p_0} g(\dot\zeta(t),\dot\zeta(t))\,\diff t
=\frac{1}{mp+1}\int_0^{mp+1} g(\dot\zeta(t),\dot\zeta(t))\,\diff t
=E^{mp+1}(\zeta),\\
\forall \zeta\in\Omega,\ m\in\K.
\end{align*}
Examples of such  $\Omega$'s are the manifolds $\Lambda^{m_0,\mu_0}$ and $\mpath{\fix(\iso^{\alpha_0})}{\iso^{\theta_0}}{\tau_0}$ for suitable value of the integer parameters and for suitable $\K$, see Section~\ref{s:iterated_local_homology}. 

The next statements in this section were originally established by Bangert and Klingenberg \cite{Bangert_Klingenberg:Homology_generated_by_iterated_closed_geodesics} in the special case $I=\id$.

\begin{lem}\label{l:bangert}
Consider $c>0$, and let $\Omega'$ be the union of the connected components of $\Omega$ having  non-empty intersection with the sub-level $\{E^{p_0}|_\Omega<c\}$. For each $m\in\K$, consider the homomorphism
\[
\iota^m_*:\Hom_*(\Omega',\{E^{p_0}|_\Omega<c\})\to\Hom_*(\mpathMI{mp+1},\{E^{mp+1}<c\})
\]
induced by the inclusion. Then, for each $h\in\Hom_*(\Omega',\{E^{p_0}|_\Omega<c\})$ we have $\iota^m_*(h)=0$ provided $m$ is large enough.
\end{lem}

\begin{proof}
The statement is straightforward in homological degree zero, since the group $\Hom_0(\Omega',\{E^{p_0}|_\Omega<c\})$ is trivial. Now, consider a non-zero  $h\in\Hom_d(\Omega',\{E^{p_0}|_\Omega<c\})$ for some $d\geq 1$ (if it exists). Let $\mu$ be a relative cycle representing  $h$. We denote by $\Sigma^j(\mu)$ the set of singular simplexes that are $j$-dimensional faces of the simplexes in the chain $\mu$. 

By modifying $\mu$ within the same homology class $h$ if necessary, we can assume that all the 0-simplexes in $\Sigma^0(\mu)$ are already contained in the sublevel $\{E^{p_0}|_\Omega<c\}$. We set $m_0:=1$ and, for all $\kappa\in\Sigma^0(\mu)$, we define the maps $h_\kappa',h_\kappa'':[0,1]\times\Delta^0\to\Omega$ as $h_\kappa'(t,z)=h_\kappa''(t,z):=\kappa(z)$. Here, $\Delta^0$ is the $0$-dimensional standard simplex, i.e.\ a point.

For each degree $j\in\{1,...,d\}$ we will find $m_j\in\N$ and, for each $\sigma\in\Sigma^j(\mu)$, two homotopies 
\begin{align}
\label{e:homological_bangert_1} &h_\sigma':[0,1]\times\Delta^j\to\mpath{M}{\id}{m_j p_0}\\
\label{e:homological_bangert_2} &h_\sigma'':[0,1]\times\Delta^j\to\Omega
\end{align}
with the following properties.
\begin{itemize}
\item[(i)] $m_{j-1}$ divides $m_{j}$,
\item[(ii)] $h_\sigma'(0,\cdot)=h_\sigma''(0,\cdot)=\sigma$,
\item[(iii)] $h_\sigma'(s,z)(0)=h_\sigma''(s,z)(0)$ for all $(s,z)\in[0,1]\times\Delta^j$,
\item[(iv)] $E^{m_jp_0}(h_\sigma'(1,z))<c$ for all $z\in\Delta^j$,
\item[(v)] $h_\sigma'(s,f_k(z))=h_{\sigma\circ f_k}'(s,z)$ and $h_\sigma''(s,f_k(z))=h_{\sigma\circ f_k}''(s,z)$ for all $k\in\{0,...,j\}$, where $f_k:\Delta^{j-1}\to\Delta^j$ is the standard affine map onto the $k$-th face of $\Delta^j$.
\end{itemize}

Let us assume that we have such a family of homotopies. Let $m''\in\K$ be a (large) integer that we will fix later, and let $m'$ be the maximal multiple of $m_d$ such that $m'p_0$ is less than of equal to $m''p+1$. For each  $j\geq1$ and singular simplex $\sigma\in\Sigma^j(\mu)$ we define the homotopy
\begin{align*}
&h_\sigma:[0,1]\times\Delta^j\to\mpathMI{m''p+1}
\end{align*}
by 
\begin{align}\label{e:definition_bangert_homotopy_final}
h_\sigma(s,z)(t)
=
\left\{
  \begin{array}{lll}
    h_\sigma'(s,z)(t) & & t\in[0,m'p_0], \\ \\
    h_\sigma''(s,z)(t) & & t\in[m'p_0,m''p+1].
  \end{array}
\right.
\end{align}
Notice that $h_\sigma(s,z)$ is a well-defined curve in $\mpathMI{m''p+1}$. Indeed
\begin{align*}
h_\sigma'(s,z)(m'p_0)=h_\sigma'(s,z)(0)=h_\sigma''(s,z)(0)=h_\sigma''(s,z)(m'p_0),
\end{align*}
and thus $h_\sigma(s,z)|_{[0,m''p+1]}$ is a continuous curve obtained by joining the $\W$ curves $h_\sigma'(s,z)|_{[0,m'p_0]}$ and $h_\sigma''(s,z)|_{[m'p_0,m''p+1]}$. Moreover, since $h_\sigma''(s,z)\in\Omega$, we have $\iso(h_\sigma''(s,z)(0))=h_\sigma''(s,z)(m''p+1)$, and therefore
\begin{align*}
\iso(h_\sigma'(s,z)(0))=I(h_\sigma''(s,z)(0))=h_\sigma''(s,z)(m''p+1),
\end{align*}
which proves that $h_\sigma(s,z)$ is an $\iso$-invariant curve with time-shift $m''p+1$.

The energy of $h_\sigma(1,z)$ can be estimated as follows:
\begin{align*}
E^{m''p+1}(h_\sigma(1,z))
& =
\frac{1}{m''p+1}
\bigg(
\int_0^{m'p_0}
g( \tfrac{\diff}{\diff t}h_\sigma'(1,z)(t),\tfrac{\diff}{\diff t}h_\sigma'(1,z)(t) )\,\diff t\\
& \s\s\s\s\s\s+
\int_{m'p_0}^{m''p+1}
g( \tfrac{\diff}{\diff t}h_\sigma''(1,z)(t),\tfrac{\diff}{\diff t}h_\sigma''(1,z)(t) )\,\diff t
\bigg)\\
& \leq
\frac{1}{m''p+1}
\Big(
m'p_0 E^{m'p_0}
(h_\sigma'(1,z))
+
m_d\, p_0 E^{m_d p_0}
(h_\sigma''(1,z))
\Big)\\
& =
\frac{1}{m''p+1}
\Big(
m'p_0 \underbrace{E^{m_j p_0}
(h_\sigma'(1,z))}_{\leq c-\epsilon_\sigma}
+
m_d\, p_0 \underbrace{E^{m_d p_0}
(h_\sigma''(1,z))}_{\leq c_\sigma}
\Big)\\
& \leq
c
-
\epsilon_\sigma
+
\frac{m_d\, p_0\, c_\sigma}{m''p+1},
\end{align*}
where $\epsilon_\sigma>0$ is a quantity given by condition~(iv), and 
\[c_\sigma:=\max_{z\in\Delta^j} E^{m_d p_0}(h_\sigma''(1,z)).\] 
Therefore, if we choose $m''\in\K$ large enough, for each $j\geq1$ and singular simplex $\sigma\in\Sigma^j(\mu)$, we have
\begin{align*}
E^{m''p+1}(h_\sigma(1,z))<c.
\end{align*}

By the homotopy invariance property for representatives of relative homology classes (see \cite[Lemma~1]{Bangert_Klingenberg:Homology_generated_by_iterated_closed_geodesics} or \cite[page~146]{Mazzucchelli:Critical_point_theory_for_Lagrangian_systems}), the existence of the family of homotopies defined in~\eqref{e:definition_bangert_homotopy_final} implies that $\mu$, seen as a relative chain in the pair $(\mpathMI{m''p+1},\{E^{m''p+1}<c\})$, is homologous to a relative chain contained in the sub level $\{E^{m''p+1}<c\}$. In particular  $[\mu]=0$ in $\Hom_d(\mpathMI{m''p+1},\{E^{m''p+1}<c\})$.

In order to complete the proof we only have to find suitable integers $m_j$ and construct the homotopies~\eqref{e:homological_bangert_1} and~\eqref{e:homological_bangert_2} satisfying properties (i--v). We do this inductively on the homological degree of the involved singular simplexes, starting in degree 1 and going up with the dimension.

Let $\sigma\in\Sigma^1(\mu)$. Notice that $\Delta^1=[0,1]$ and $\Omega \subset \mpath{M}{\id}{p_0}$. Thus, let us consider $\sigma$ as a continuous path of the form
\[\sigma:[0,1]\to\mpath{M}{\id}{p_0}.\] 
By our assumptions on the elements of $\Sigma^0(\mu)$, we have $E^{p_0}(\sigma(z))<c$ for $z=0$ and $z=1$. Let $m$ be a positive integer that we will fix later, and consider the continuous path
\[\sigma_{\langle m\rangle}:[0,1]\to\mpath{M}{\id}{m p_0}\]
obtained by applying Bangert's construction of Section~\ref{s:bangert_construction} to $\sigma$. We define the map $h_\sigma':[0,1]\times\Delta^1\to\mpath{M}{\id}{m p_0}$ to be a homotopy as in equation~\eqref{e:homotopy_path_bangert_path}, i.e.
\begin{align*}
h_\sigma'(s,z)=
\left\{
  \begin{array}{lcl}
    (\sigma|_{[0,s]})_{\langle m\rangle}(z) &  & \mbox{ if }z\in[0,s], \\\\ 
    \sigma(z) &  & \mbox{ if }z\in[s,1]. 
  \end{array}
\right.
\end{align*}
In particular $h_\sigma'(0,\cdot)=\sigma$,  $h_\sigma'(1,\cdot)=\sigma_{\langle m\rangle}$, and $h_\sigma'(s,z)=\sigma(z)$ for all $s\in[0,1]$ and $z\in\partial\Delta^1=\{0,1\}$. We define  $h_\sigma'':[0,1]\times\Delta^1\to\Omega$ to be the unique map such that 
\[
h_\sigma''(s,z)|_{[0,p_0]}=h_\sigma'(s,z)|_{[0,p_0]},\s\s\forall s\in[0,1],\ z\in\Delta^1.
\]
By an estimate as in~\eqref{e:estimate_bangert}, if we fix $m\in\N$ large enough, for all $\sigma\in\Sigma^1(\mu)$ and $z\in\Delta^1$ we obtain $E^{m p_0}(\sigma_{\langle m\rangle}(z))<c$. Thus, we set $m_1:=m$. Notice that $h_\sigma'$, $h_\sigma''$ and $m_1$ satisfy assumptions (i--v) listed above when $j=1$.

Now, let us proceed iteratively with the construction: assuming we are done up to degree $j-1$,  we show how to make the next step for $j$-simplexes in $\Sigma^{j}(\mu)$. We consider $m\in\N$ that is a (large) multiple of $m_{j-1}$, and we will fix it later. Up to a minor modification in the previous steps, we can assume that, for each $i<j$, $\kappa\in\Sigma^i(\mu)$ and $s\in[\tfrac12,1]$, we have $h_\kappa'(s,\cdot)=h_\kappa'(1,\cdot)$ and  $h_\kappa''(s,\cdot)=h_\kappa''(1,\cdot)$. Let $\sigma\in\Sigma^j(\mu)$. We begin by putting together the homotopies of the faces of $\sigma$, in such a way that we obtain a continuous map
\[
h_{\partial\sigma}':[0,1]\times\partial \Delta^j\to\mpath{M}{\id}{m p_0}.
\]
Notice that $E^{mp_0}(h_{\partial\sigma}'(s,z))<c$ for all $s\in[0,1]$ and $z\in\partial\Delta^j$. Moreover $h_{\partial\sigma}'(s,\cdot)=h_{\partial\sigma}'(1,\cdot)$ for all $s\in[\tfrac12,1]$. Now consider a retraction 
\[
r:[0,\tfrac12]\times\Delta^j\to([0,\tfrac12]\times\partial \Delta^j)\cup(\{0\}\times\Delta^j).
\]
We define 
\[h_\sigma':[0,\tfrac12]\times\Delta^j\to\mpath{M}{\id}{m_{j-1} p_0}\subset\mpath{M}{\id}{m  p_0}\] 
by $h_\sigma':=h_{\partial\sigma}'\circ r$. Set $\tilde\sigma:=h_\sigma'(\tfrac12,\cdot)$, and see it as a map of the form 
\[\tilde\sigma:\Delta^j\to\mpath{M}{\id}{m_{j-1} p_0}.\]
As in \cite[page~461]{Long:Multiple_periodic_points_of_the_Poincare_map_of_Lagrangian_systems_on_tori}, let $\Line\subseteq\R^j$ be the 1-dimensional vector subspace generated by the vector pointing to the  barycenter of standard $j$-simplex $\Delta^j\subset\R^j$. For each $s\in[0,1]$ we denote by $s\Delta^j$ the rescaled $j$-simplex given by $\{s z\,|\, z\in\Delta^j\}$. For each $z\in s\Delta^j$, we define $[a(s,z),b(s,z)]$ to be the maximum segment inside $s\Delta^j$ that contains $z$ and is parallel to $\Line$. Notice that $a$ and $b$ are continuous functions on their domains. We define the other piece of homotopy
\[h_\sigma':[\tfrac12,1]\times\Delta^j\to\mpath{M}{\id}{m  p_0}\] 
by
\begin{align*}
h_\sigma'(s,z)
=
\left\{
  \begin{array}{lll}
    (\tilde\sigma|_{[a(2s-1,z),b(2s-1,z)]})_{\langle m \rangle}(z) &  & \mbox{if }z\in (2s-1)\Delta^j,\\\\ 
    \tilde\sigma(z) &  & \mbox{if }z\not\in (2s-1)\Delta^j. 
  \end{array}
\right.
\end{align*}
Basically here we are piecing together Bangert's homotopies (described in Section~\ref{s:bangert_construction}) of each  path $\tilde\sigma|_{[a(1,z),b(1,z)]}:[a(1,z),b(1,z)]\to\mpath{M}{\id}{m_{j-1} p_0}$.  We define  $h_\sigma'':[0,1]\times\Delta^1\to\Omega$ to be the unique map such that 
\[
h_\sigma''(s,z)|_{[0,p_0]}=h_\sigma'(s,z)|_{[0,p_0]},\s\s\forall s\in[0,1],\ z\in\Delta^j.
\]

If we fix $m\in\N$ to be a sufficiently large multiple of $m_{j-1}$, an estimate as in~\eqref{e:estimate_bangert} implies that $E^{mp_0}(h_\sigma'(1,z))<c$  for all $\sigma\in\Sigma^j(\mu)$ and $z\in\Delta^j$. We set $m_j:=m$. As for the case in degree~1, we have that $h_\sigma'$, $h_\sigma''$ and $m_j$ satisfy assumptions (i--v) listed above.
\end{proof}

From   Lemma~\ref{l:bangert} we can infer an analogous statement concerning the local homology of periodic $\iso$-invariant geodesics.

\begin{lem}\label{l:bangert_local_homology}
Let $\gamma$ be a critical point of $E$ that is periodic of basic period $p\geq1$. Assume that there exist a degree $d\geq 2$ and an infinite set $\K\subset\N$ such that, for all $m\in\K$, the local homology $\Loc_d(E,\orb(\gamma^{mp+1}))$ is non-trivial. Then, there are arbitrarily large $m\in\K$ such that the homomorphisms
\begin{align*}
j^m_*:\Loc_d(E,\orb(\gamma^{mp+1}))\to\Hom_d(\pathMI,\{E<E(\gamma^{mp+1})\}) 
\end{align*}
induced by the inclusion are not injective.
\end{lem}

\begin{proof}
Since the local homology $\Loc_d(E,\orb(\gamma^{mp+1}))$ is non-trivial, we infer that $\ind(\gamma^{mp+1})\leq d$ for all $m\in\K$, and therefore $\avind(\gamma)=0$. Thus we can apply the results of  Section~\ref{s:iterated_local_homology}. Let us employ the (equivalent) variational setting introduced there. Our assumption on the local homology can be rephrased by saying that   $\Loc_d(E^{mp+1},\orb(\gamma))$ is non-trivial for all $m\in\K$. Let us consider separately the cases in which the period $p$ of $\gamma$ is irrational or rational.

If $p$ is irrational, let us consider the bounded function $\mu:\N\to\N$ of Lemma~\ref{l:irrational}. By the pigeonhole principle we can find $m_0,\mu_0\in\N$ and an infinite subset $\K'\subset\K$ such that $m\equiv m_0\mod\mu_0$ and $\mu(m)=\mu_0$ for all $m\in\K'$.  As we remarked right before Lemma~\ref{l:irrational}, we have that $\Lambda^{m,\mu_0}=\Lambda^{m_0,\mu_0}\subset\mpathMI{mp+1}$, and $E^{\mu_0p}=E^{mp+1}$ on this space. Lemma~\ref{l:irrational} implies that the inclusion induces an isomorphism
\[
\Loc_*(E^{\mu_0p},\orb(\gamma))\toup^{\cong}\Loc_*(E^{mp+1},\orb(\gamma)),\s\s\forall m\in\K'.
\]
We set $p_0:=\mu_0p$, and recall that every $\zeta\in\Lambda^{m_0,\mu_0}$ is a $p_0$-periodic curve. We also set $\Omega:=\Lambda^{m_0,\mu_0}$.

In the other case, when $p$ is rational, we apply the pigeonhole principle to the bounded  functions $\tau,\alpha,\theta:\N\to\N$ of Lemma~\ref{l:rational}. We find $\tau_0,\alpha_0,\theta_0\in\N$ and an infinite subset $\K'\subset\K$ such that $\tau(m)=\tau_0$, $\alpha(m)=\alpha_0$ and $\theta(m)=\theta_0$ for all $m\in\K'$. We recall that, for all $m\in\K'$, we have that  $\mpath{\fix(\iso^{\alpha_0})}{\iso^{\theta_0}}{\tau_0}\subset\mpathMI{mp+1}$, and the functions $E^{\tau_0}$ and $E^{mp+1}$ coincide on $\mpath{\fix(\iso^{\alpha_0})}{\iso^{\theta_0}}{\tau_0}$. Lemma~\ref{l:rational} implies that the inclusion induces an isomorphism
\[
\Loc_*(E^{\tau_0}|_{\mpath{\fix(\iso^{\alpha_0})}{\iso^{\theta_0}}{\tau_0}},\orb(\gamma))\toup^{\cong}\Loc_*(E^{mp+1},\orb(\gamma)),\s\s\forall m\in\K'.
\]
We set $p_0:=\tau_0\alpha_0$. As before, every $\zeta\in\mpath{\fix(\iso^{\alpha_0})}{\iso^{\theta_0}}{\tau_0}$ is a $p_0$-periodic curve, and we set $\Omega:=\mpath{\fix(\iso^{\alpha_0})}{\iso^{\theta_0}}{\tau_0}$.

In either case ($p$ rational or irrational), $\Omega$ satisfies the assumptions required in  Lemma~\ref{l:bangert}. Let $\Omega'$ be the union of the connected components of $\Omega$ that intersect the sub-level $\{E^{p_0}<c\}$, where $c=E^{p_0}(\gamma)$. Notice that the orbit of $\gamma$ is not a local minimum of the energy $E^{p_0}|_{\Omega}$, otherwise we would have 
\[\Loc_d(E^{p_0}|_{\Omega},\orb(\gamma))=\Hom_d(\orb(\gamma))\simeq\Hom_d(S^1)=0,\]
whereas $\Loc_d(E^{p_0}|_{\Omega},\orb(\gamma))$ is non-trivial by our assumptions. Therefore $\orb(\gamma)$ is contained in $\Omega'$, and we have the following commutative diagram where all the homomorphisms are induced by inclusions.
\begin{align*}
\xymatrix{
\Loc_d(E^{p_0}|_{\Omega},\orb(\gamma))
\ar[rr]^{\cong}
\ar[d]
&&
\Loc_d(E^{mp+1},\orb(\gamma))
\ar[d]^{\tilde j^m_*}
\\
\Hom_d(\Omega',\{E^{p_0}|_{\Omega}<c\})
\ar[rr]^{\iota^m_*\ \ \ \ }
&&
\Hom_d(\mpathMI{mp+1},\{E^{mp+1}<c\})
}
\end{align*}
Our statement follows from Lemma~\ref{l:bangert}.
\end{proof}

After these preliminaries, we can now state and prove the main result of this section.

\begin{prop}\label{p:bangert_klingenberg}
Let $\gamma$ be a critical point of $E$ that is periodic of basic period $p\geq1$. If there exists a degree $d\geq2$ such that, for infinitely many $m\in\N$, the local homology $\Loc_d(E,\orb(\gamma^{mp+1}))$ is non-trivial, then the Riemannian manifold $(M,g)$ contains infinitely many $\iso$-invariant geodesics.
\end{prop}

\begin{proof}
We prove the proposition by contradiction, assuming that there are only finitely many $\iso$-invariant geodesics $\ell_1,...,\ell_r$. By \cite[Theorem~2.4]{Grove:Isometry_invariant_geodesics}, all these geodesics must be closed. Let $\gamma_i:\R\to M$ be a parametrization of $\ell_i$ with constant speed, period $p_i\geq1$, and such that $\iso(\gamma_i(t))=\gamma_i(t+1)$ for all $t\in\R$. Thus the critical orbits of the energy $E$ corresponding to $\ell_i$ are all the $\orb(\gamma_i^{mp_i+1})$, for $m\in\N$. We set
\begin{align*}
D_i=\big\{d\in\N\ \big|\ \Loc_d(E,\orb(\gamma_i^{mp_i+1}))\neq0\mbox{ for infinitely many $m\in\N$}\big\},\\ \forall i=1,...,r. 
\end{align*}
Up to renaming the $\iso$-invariant geodesics, we can assume that $D_i\neq\varnothing$ for $i=1,...,s$, whereas $D_i=\varnothing$ for all $i>s$. Notice that, by the assumptions of the theorem, $s\geq1$. 
We also set
\[
d_i=\max D_i,\s\s \forall i=1,...,s.
\]
Notice that $d_i$ is finite. Indeed, the fact that the local homology $\Loc_*(E,\orb(\gamma_i^{mp_i+1})$ is non-trivial in a fixed degree for arbitrarily large $m$ implies that $\avind(\gamma_i)=0$. Therefore, by Proposition~\ref{p:average_index}, $\ind(\gamma_i^{mp_i+1})=0$ for all $m\in\N$, and we conclude that $d_i\leq\nul(\gamma_i^{mp_i+1})+1\leq2\dim(M)$.

Let $i$ be such that $d_i=\max\{d_1,...,d_s\}$. In particular $d_i\geq d\geq2$, where $d$ is the integer in the statement of the proposition. We reset $\gamma:=\gamma_i$, $p:=p_i$ and $d:=d_i$. Notice that this new $\gamma$ still satisfies the assumptions of the proposition (with respect to the new $p$ and $d$). Fix $\overline m\in\N$ sufficiently large so that, for all $j>s$, we have the following: 
\begin{itemize}
\item if $\avind(\gamma_j)=0$ then $\Loc_*(E,\orb(\gamma_j^{mp_j+1}))$ is trivial for all integers $m\geq\overline{m}$,
\item if $\avind(\gamma_j)>0$ then $\ind(\gamma_j^{mp_j+1})>d+2$  for all integers $m\geq\overline{m}$. 
\end{itemize}
We set
\[
\overline{e}:=\max\big\{ E(\gamma_j^{\overline{m}p_j+1})\ \big|\ s<j\leq r \big\}.
\]
For all $j>s$, the inequality $E(\gamma_j^{mp_j+1})>\overline{e}$ implies $m>\overline{m}$. Choose another large enough $\overline{m}'\geq\overline{m}$ such that $E(\gamma^{mp+1})>\overline{e}$ for all integers $m\geq\overline{m}'$.

Now, fix an integer $m\geq\overline{m}'$, and set $c=c(m):=E(\gamma^{mp+1})$. Take $\epsilon=\epsilon(m)>0$ small enough so that the interval $(c,c+\epsilon]$ does not contain critical points of $E$. In particular, the inclusion induces an injective homomorphism
\begin{align}\label{e:technical_emb_loc_hom_epsilon}
 \Loc_*(E,\orb(\gamma^{mp+1}))
 \hookrightarrow
 \Hom_*(\{E<c+\epsilon\},\{E<c\}).
\end{align}
Moreover, since $E$ does not have any critical orbit with critical value larger than $\overline{e}$ and non-zero local homology in degree $d+1$, the Morse inequalities imply
\begin{align}
\label{e:technical_relative_homology_zero_epsilon}
\Hom_{d+1}\big(\{E<c+\epsilon\},\{E<c\}\big)=0,\\
\label{e:technical_relative_homology_zero}
\Hom_{d+1}\big(\pathMI,\{E<c+\epsilon\}\big)=0.
\end{align}
By \eqref{e:technical_emb_loc_hom_epsilon}, \eqref{e:technical_relative_homology_zero_epsilon}, and the long exact sequence of the triple
\[
\{E<c\}\ \subset\ \{E<c\}\cup\orb(\gamma^{mp+1})\ \subset\ \{E<c+\epsilon\}
\]
we infer that $\Hom_{d+1}(\{E<c+\epsilon\},\{E<c\}\cup\orb(\gamma^{mp+1}))$ is trivial. This, together with~\eqref{e:technical_relative_homology_zero} and the long exact sequence of the triple
\[
\{E<c\}\cup\orb(\gamma^{mp+1})\ \subset\ \{E<c+\epsilon\}\ \subset\ \pathMI,
\]
implies that $\Hom_{d+1}(\pathMI,\{E<c\}\cup\orb(\gamma^{mp+1}))$ is trivial. Finally this, together with the long exact sequence of the triple 
\[
\{E<c\}
\ \subset\
\{E<c\}\cup\orb(\gamma^{mp+1})
\ \subset\ 
\pathMI,
\]
implies that the inclusion induces an injective homomorphism
\begin{align*}
\Loc_d(E,\orb(\gamma^{mp+1}))\hookrightarrow\Hom_d(\pathMI,\{E<E(\gamma^{mp+1})\}).
\end{align*}
Since this is true for any integer $m\geq\overline{m}'$, it contradicts the assertion of Lemma~\ref{l:bangert_local_homology}.
\end{proof}

\section{Proof of Theorem~\ref{t:main}}\label{s:proof_main}

The proof of Theorem~\ref{t:main} goes along the following lines. By applying Morse theory to the energy function $E$, we find infinitely many critical points that correspond  to either infinitely many (geometrically distinct) $\iso$-invariant geodesics, or to a single $\iso$-invariant geodesic satisfying the assumptions of Proposition~\ref{p:bangert_klingenberg}. However, this latter proposition implies that there are infinitely many $\iso$-invariant geodesics. We go over this argument in the following.

Since $\rank\Hom_1(M_1)\neq0$, we can find a smooth 1-cycle $s$ such that 
\begin{align}\label{e:property_of_s}
s^m\neq0\ \mbox{ in }\ \Hom_1(M_1),\s\s \forall m\in\N.  
\end{align}
Let $\sigma:\R\to M_1$ be a 1-periodic smooth curve such that $\sigma|_{[0,1]}$ is a parametrization of $s$. For $m\in\N$, we define a smooth map $\Sigma_m:M_2\to\path{M}{\id}$ by $\Sigma_m(q)(t)=(\sigma^m(t),q)$. We denote by $C_m$ the connected component of the free loop space $\path{M}{\id}$ containing $\Sigma_m(M_2)$. By~\eqref{e:property_of_s} we have that $C_m\cap C_n=\varnothing$ if $m\neq n$.

Consider a smooth homotopy $\iso_t:M\to M$, where $t\in[0,1]$, such that $\iso_0=\id$ and $\iso_1=\iso$. This homotopy induces a continuous map $\iota:\path{M}{\id}\to\pathMI$ in the following way. For all $\zeta\in\path{M}{\id}$ and $t\in[0,1]$ we set
\begin{align*}
\iota(\zeta)(t)
=
\left\{
  \begin{array}{ll}
    \zeta(2t) & \mbox{ if }t\in[0,1/2], \\\\
    I_{2t-1}(\zeta(0)) & \mbox{ if }t\in[1/2,1], \\ 
  \end{array}
\right.
\end{align*}
and we extend $\iota(\zeta)$ to the whole real line in such a way that \[\iso(\iota(\zeta)(t))=\iota(\zeta)(t+1),\s\s\forall t\in\R.\]
It is easy to see that $\iota$ is a homotopy equivalence (see \cite[Lemma~3.6]{Grove:Condition_C_for_the_energy_integral_on_certain_path_spaces_and_applications_to_the_theory_of_geodesics}). If we denote by $D_m$ the connected component of $\pathMI$ containing $\iota(C_m)$, we have that $D_m\cap D_n=\varnothing$.

Let $\ev:D_m\to M_2$ be the evaluation map given by $\ev(\zeta)=\zeta_2(0)$ for all $\zeta\in D_m$, where $\zeta_2$ denotes the $M_2$-factor of the curve $\zeta:\R\to M_1\times M_2$. Notice that $\ev\circ\iota\circ\Sigma_m$ is the identity on $M_2$. Therefore $\ev$ is a left inverse for $\iota\circ\Sigma_m$, and this latter map  induces an injective homomorphism
\begin{align*}
(\iota\circ\Sigma_m)_*: \Hom_*(M_2)\hookrightarrow\Hom_*(D_m).
\end{align*}

We can assume that $M_2$ is an orientable manifold. If this is not true, we proceed as follows. We replace $M_2$ by its orientable 2-fold covering $\widetilde M_2$. Our manifold $M$ admits a $2$-fold covering $\widetilde M$ that is homeomorphic to the 2-fold covering $M_1\times\widetilde M_2$ of $M_1\times M_2$. It suffices to lift the Riemannian metric $g$ and the isometry $\iso$ to $\widetilde M$, and  carry over the proof of Theorem~\ref{t:main} for $(\widetilde M,\widetilde g)$. Indeed, infinitely many $\widetilde\iso$-invariant geodesics on $\widetilde M$ project down to infinitely many $\iso$-invariant geodesics on $M$.

By our orientability assumption on $M_2$, we have that $\Hom_d(M_2)$ is non-trivial for $d=\dim(M_2)$. Thus $\Hom_d(D_m)$ is non-trivial as well. Since we are looking for infinitely many $\iso$-invariant geodesics, we can assume that all the critical orbits of $E$ are isolated (otherwise we are already done). By the Morse inequalities, there exists a critical point $\gamma_m$ contained in the connected component $D_m$ and such that the local homology $\Loc_d(E,\orb(\gamma_m))$ is non-trivial.

Summing up, we have found infinitely many critical orbits $\orb(\gamma_m)$, where $m\in\N$, whose local homology is non-trivial in the fixed degree $d=\dim(M_2)\geq2$. If these critical orbits correspond to only finitely many (geometrically distinct) $\iso$-invariant geodesics, then there must be a critical point $\gamma$ of $E$ which is periodic with period $p\geq1$ and such that, for infinitely many $m\in\N$, the critical orbit $\orb(\gamma_m)$ is the critical orbit $\orb(\gamma^{\mu p+1})$ for some $\mu=\mu(m)$. In particular $\gamma$ satisfies the assumptions of Proposition~\ref{p:bangert_klingenberg}, and therefore $(M,g)$ must have  infinitely many $\iso$-invariant geodesics. \hfill\qed


\bibliography{_biblio}
\bibliographystyle{amsalpha}

\vspace{0.5cm}

\end{document}